\documentclass[11pt]{article}
\usepackage[english]{babel}
\usepackage{amsmath,amsthm}
\usepackage{amsfonts}
\usepackage[comma,numbers,square,sort&compress]{natbib}
\usepackage{graphicx,subfigure,amsmath,amssymb,mathrsfs,amsfonts,amstext,amsthm}
\usepackage{latexsym, amssymb}
\usepackage{graphicx}
\usepackage{subfigure}
\usepackage{pgf,tikz}
\usepackage{pstricks}
\newtheorem{thm}{Theorem}[section]
\newtheorem{cor}[thm]{Corollary}
\newtheorem{lem}[thm]{Lemma}
\newtheorem{prop}[thm]{Proposition}
\theoremstyle{definition}
\newtheorem{defn}[thm]{Definition}
\theoremstyle{remark}
\newtheorem{rem}[thm]{Remark}
\numberwithin{equation}{section}
\begin{document}

\title{\bfseries\textrm{Properties of Quasi-synchronization Time of High-dimensional Hegselmann-Krause Dynamics}
\footnotetext{
Wei Su, Meiru Jiang and Yongguang Yu are with School of Mathematics and Statistics, Beijing Jiaotong University, Beijing 100044, China, {\tt su.wei@bjtu.edu.cn, 23121692@bjtu.edu.cn, ygyu@bjtu.edu.cn}. Ge Chen is with National Center for Mathematics and Interdisciplinary Sciences \& Key Laboratory of Systems and
Control, Academy of Mathematics and Systems Science, Chinese Academy of Sciences, Beijing 100190,
China, {\tt chenge@amss.ac.cn}. 
}
 }
\author{Wei Su, Meiru Jiang, Yongguang Yu, Ge Chen }

%
\date{}%
\maketitle
\begin{abstract}
The behavior of one-dimensional Hegselmann-Krause (HK) dynamics driven by noise has been extensively studied. Previous research has indicated that within no matter the bounded or the unbounded space of one dimension, the HK dynamics attain quasi-synchronization (synchronization in noisy case) in finite time. However, it remains unclear whether this phenomenon holds in high-dimensional space. This paper investigates the random time for quasi-synchronization of multi-dimensional HK model and reveals that the boundedness and dimensions of the space determine different outcomes. To be specific, if the space is bounded, quasi-synchronization can be attained almost surely for all dimensions within a finite time, whereas in unbounded space, quasi-synchronization can only be achieved in low-dimensional cases (one and two). Furthermore, different integrability of the random time of various cases is proved.
\end{abstract}

\section{Introduction}\label{Section:introduction}
The Hegselmann-Krause (HK) model was first introduced in the field of opinion dynamics to describe the opinion evolution of individuals who interact with others and whose opinions are influenced by those of the people around them \cite{Hegselmann2002}. In the HK model, the individuals update their opinions over time by taking the average of the opinions of all their neighbors whose opinions are close enough to their own. This closeness is determined by a bounded confidence threshold, such that agents influence each other's opinion only if their opinions lay within the confidence threshold. Though initially proposed in the context of opinion dynamics, the HK model captures a fundamental self-organizing mechanism in complex systems. Beyond its original application, it has also been adopted as a basic game learning algorithm \cite{Acemoglu2011, Basar2015} and has found widespread use in diverse fields, including demand response programs in smart grids \cite{Nava2015} and hybrid energy storage management \cite{Abdelghany2024}.

Among the many properties, one of the interesting features of the model is that it can be synchronized by random noise. This phenomenon, also known as ``noise-induced order" in self-organizing systems, was first found in some simulation studies \cite{Mas2010,Carro2013,Pineda2013}. Then the analysis of the phenomenon was considered based on some noisy HK-type models. A proof based on the original discrete HK model in bounded state space (where the values of states are typically confined to a closed interval in $\mathbb{R}$) was obtained in \cite{Su2017auto} and a continuous case in \cite{Wang2017}. Then the cases of HK model in unbounded space (where the values of states are allowed to take any value in $\mathbb{R}$) \cite{Su2019tac} and also the Deffuant-Weisbuch (DW) model (another typical bounded confidence opinion dynamics but with asynchronous updating rule\cite{Deffuant2000}) were proved \cite{Su2022cs}. However, the above conclusions are all considered in one-dimensional space, while the question of whether this phenomenon persists in higher dimensions has remained unanswered until now.

The behavior of networked dynamical systems can change significantly with the dimensionality of the space, as highlighted by previous studies of networked systems. For instance, the synchronization of coupled oscillators is more challenging in higher dimensions, as their dynamics become more complex \cite{Parsegov2017}. Similarly, synchronization in networks of chaotic systems becomes more difficult as the dimensionality of the space increases \cite{Balanov2009}. The ``curse of dimensionality" is a well-known phenomenon in various fields like machine learning, optimization, data mining, and control \cite{Bellman1961,Aggarwal2001}. These studies suggest that the behavior of the HK dynamics in higher dimensions may be very different from that observed in one-dimensional space.

The study of high-dimensional HK models has primarily focused on the influence of spatial structure and network topology on the evolution of opinion dynamics. Synchronization, or the emergence of consensus, reflects the ability of individuals to reach agreement or form coordinated clusters of opinions, a key aspect of collective behavior. This phenomenon is a fundamental problem in the study of opinion dynamics models, including the multi-dimensional HK model. For instance, Lorenz explored various factors affecting the evolution of high-dimensional HK models in deterministic settings (see \cite{Lorenz2007, Lorenz2008} and references therein). Further studies on multi-dimensional HK dynamics and their variants can be found in \cite{Basar2015, Nedic2012, Parsegov2017, Chazelle2017, Pasquale2022}, as well as in \cite{Bernardo2024} (and references therein). In the stochastic case, Bolychev et al. investigated the probability laws governing noisy high-dimensional HK models \cite{Bolychev2021}.
Despite extensive research on high-dimensional HK dynamics and synchronization under various conditions, the phenomenon of synchronization under the influence of noise in high-dimensional spaces remains unexplored. This gap motivates our study, which aims to provide a deeper understanding of how noise shapes synchronization in high-dimensional systems.

In this paper, we extend previous findings on one-dimensional HK dynamics by examining the behavior of multi-dimensional HK models driven by noise. Our main focus is on the stopping time, which is the time instance when the system reaches quasi-synchronization. By investigating the integrability of stopping time, we identify different phenomena that occur as a function of both the boundedness and the dimensionality of the space. Specifically, we find that when the space is bounded, the systems can reach quasi-synchronization in finite time for all dimensions, and the stopping time is integrable. In unbounded space, the systems only reach quasi-synchronization for low-dimensional systems (typically one- and two-dimension) and the stopping time is not integrable. In higher-dimensional spaces, however, the presence of noise can only result in a synchronization with a positive probability. The results fully reveal that noisy HK dynamics behave quite differently in high-dimensional space.

Moreover, as a fundamental mechanism of self-organizing processes, our findings on unbounded HK dynamics highlight the critical role of space boundedness in the emergence of self-organizing order. This aligns with experimental studies in other fields, where boundedness has been shown to be essential for the formation of ordered structures \cite{Martin2003, Hanczyc2003, Gu2025}. Thus, our work not only advances the theoretical understanding of HK dynamics but also provides insights into the broader principles governing self-organization in complex systems.

The rest of the paper is organized as follows. In the next section, we provide an introduction of the HK model in high-dimensional spaces and some preliminary definitions. Section \ref{Section:Results} describes our main results in different dimensional spaces. In Section \ref{Section:Conclusions}, we conclude the paper and discuss the implications of our findings for future research in the dynamics of networked systems in high-dimensional spaces.

\section{Model and preliminaries}\label{Section:high-dim HK}
This paper considers the noisy HK model in bounded and unbounded spaces, called as \emph{bounded noisy HK model} and \emph{unbounded noisy HK model} respectively.
Denote $\mathcal{V}=\{1,2,\ldots,n\}$ as the set of $n$ agents, $x_i(t)\in(-\infty,\infty)^d$ or $[-1,1]^d, i\in\mathcal{V}, t\geq 0$ be the state of agent $i$ at time $t$ in bounded or unbounded space,
where $n\geq 2$ and $d\geq 1$ are two integers. Let
$\epsilon$ be a positive constant representing the confidence threshold (or neighbor radius) of agents, and
\begin{equation}\label{Model:neighbor}
 \mathcal{N}_i(t)=\{j\in\mathcal{V}\; \big|\; \| x_j(t)-x_i(t)\| \leq \epsilon\}
\end{equation}
be the set of neighbors of agent $i$ at time $t$, where $\|\cdot\|$ is the Euclidean norm.
Denote for $y\in \mathbb{R}$
\begin{equation*}
  (y)_{[-1,1]}=\left\{
              \begin{array}{ll}
                -1, & \hbox{$y<-1$} \\
                y, & \hbox{$-1\leq y\leq 1$,} \\
                1, & \hbox{$y>1$}
              \end{array}
            \right.
\end{equation*}
and $$(y_1,...,y_d)^{\top}_{[-1,1]}=\big(y_1)_{[-1,1]},...,(y_d)_{[-1,1]}\big)^{\top},$$
then the update rules for the bounded and the unbounded noisy HK dynamics can be formulated as: for all $ i\in\mathcal{V}, t\geq 0$,
\begin{equation}\label{Model:HKmodelbound}
x_i(t+1)=\bigg(\frac{1}{|\mathcal{N}_i(t)|}\sum\limits_{j\in \mathcal{N}_i(t)}x_j(t)+\xi_i(t+1)\bigg)_{[-1,1]}
\end{equation}
and
\begin{equation}\label{Model:HKmodelunbound}
x_i(t+1)=\frac{1}{|\mathcal{N}_i(t)|}\sum\limits_{j\in \mathcal{N}_i(t)}x_j(t)+\xi_i(t+1),
\end{equation}
respectively,
where $|\cdot|$ is the cardinality of a set. Here, we make assumptions on the confidence threshold $\epsilon$ for both bounded and unbounded cases. For the bounded model (\ref{Model:HKmodelbound}), $\epsilon$ is assumed to lie within the scale of the space, that is the largest distance of two points in $[-1,1]^d$, specifically $\epsilon \in (0, 2\sqrt{d}]$,  where $2\sqrt{d}$ is the maximum Euclidean distance between two points in the $d$-dimensional hypercube $[-1, 1]^d$. This ensures that interactions remain localized and meaningful within the bounded system. For the unbounded model (\ref{Model:HKmodelunbound}), $\epsilon$ can take any positive value, i.e., $\epsilon \in (0, \infty)$.  Formally, we assume
\begin{equation}\label{Model:confithre}
\begin{split}
  &\epsilon\in(0,2\sqrt{d}], \,\quad\text{for bounded system (\ref{Model:HKmodelbound}),} \\
  &\epsilon\in(0,\infty), \qquad\text{for unbounded system (\ref{Model:HKmodelunbound})}.
\end{split}
\end{equation}
In addition, $\{\xi_i(t), i\in\mathcal{V}, t>0\}$ are i.i.d. random noises satisfying $\mathbb{E}\, \xi_i(t)=\textbf{0}$, $\mathbb{E}\, \xi_i^2(t)\neq\textbf{0}$ and $\|\xi_i(t)\|\leq\delta$ a.s. with $\delta>0$.
In the context of opinion dynamics, noise captures the randomness of information exposure through social media and other communication channels. It reflects the unpredictable nature of how individuals encounter and process information, as well as the internal variability in their responses. This interpretation aligns with the complex and stochastic nature of opinion formation in the digital age.

To proceed further, we need to introduce some preliminary definitions as follows.

Let $\Omega$ be the sample space of $\xi(t)=\{\xi_i(t),i\in\mathcal{V}, t\geq 1\}$, $\mathcal{F}$ be the generated $\sigma-$algebra, and $\mathbb{P}$ be the probability measure on $\mathcal{F}$, so the underlying probability space is written as $(\Omega,\mathcal{F},\mathbb{P})$.
In addition, denote $\mathbb{E}\{\cdot\}$ as the expectation of a random variable.
In what follows, the ever appearing time symbols $t$ (or $T$, etc.) will all refer to the random variables $t(\omega)$ (or $T(\omega)$, etc.), and for simplicity, they will be still written as $t$ (or $T$).

The definition of \emph{quasi-synchronization} of the noisy model (\ref{Model:neighbor})-(\ref{Model:HKmodelunbound}) is then as in \cite{Su2017auto}:
\begin{defn}\label{robconsen}
Denote
\begin{equation}\label{Equa:defdVt}
  d_{\mathcal{V}}(t)=\max\limits_{i, j\in \mathcal{V}}\|x_i(t)-x_j(t)\|,
\end{equation}
and let
\begin{equation}\label{Equa:defTquasisyn}
  T'=\inf\{t\geq 0: d_{\mathcal{V}}(t')\leq \epsilon, \forall\,\, t'\geq t\}.
\end{equation}
If $\mathbb{P}\{T'<\infty\}=1$, we say a.s. the system (\ref{Model:neighbor})-(\ref{Model:HKmodelunbound}) achieves quasi-synchronization in finite time.
\end{defn}
In Definition \ref{robconsen}, $d_{\mathcal{V}}(t)$ measures the maximum state (or opinion) difference among agents. In deterministic systems, $d_{\mathcal{V}}(t)=0$ indicates perfect consensus, but in stochastic systems, noise prevents this, making exact synchronization unattainable. Quasi-synchronization, defined as $d_{\mathcal{V}}(t)\leq\epsilon$, provides a practical measure for near-consensus in noisy environments. When quasi-synchronization is achieved, all individuals become mutual neighbors, and in the absence of noise, they would share a common opinion at the next time step. This makes quasi-synchronization a meaningful standard for assessing synchronization in real-world systems, where opinions converge to near-agreement despite random influences.

For the systems (\ref{Model:neighbor})-(\ref{Model:HKmodelunbound}), one can readily have the following property:
\begin{prop}[\cite{Su2017auto,Su2019tac}]\label{Prop:quasisynHK}
Assume  $\delta\leq \epsilon/2$ and
a.s. there exists a time  $T\geq 0$  such that $d_{\mathcal{V}}(T)\leq \epsilon$. Then, under bounded noisy HK model (\ref{Model:neighbor}) and (\ref{Model:HKmodelbound}) or unbounded noisy HK model (\ref{Model:neighbor}) and (\ref{Model:HKmodelunbound}), we have $d_{\mathcal{V}}(t)\leq \epsilon$ a.s. for all $t\geq T$.
\end{prop}
%

\section{Main Results}\label{Section:Results}
Let
\begin{equation}\label{Equa:defstopT}
 T=\inf\{t\geq 0: d_\mathcal{V}(t)\leq \epsilon\}.
\end{equation}
 It can be observed that the random time $T$ defined in (\ref{Equa:defstopT}) is not the same as $T'$ in (\ref{Equa:defTquasisyn}). $T$ in (\ref{Equa:defstopT}) is a stopping time, while $T'$ in (\ref{Equa:defTquasisyn}) is not. However, by Proposition \ref{Prop:quasisynHK}, one can obtain that, when $\delta\leq \epsilon/2$, the two random time are equivalence, i.e., $\mathbb{P}\{T' ~in~ (\ref{Equa:defTquasisyn})=T~ in~ (\ref{Equa:defstopT})\}=1$. Thus in the following, we only consider the stopping time $T$ in (\ref{Equa:defstopT}).

\subsection{Case of bounded space}\label{Subsec:boundedspace}

In this part, we consider the random time for the system in bounded space. We have the following result.
\begin{thm}\label{Thm:mainbounded}
Consider the bounded HK system (\ref{Model:neighbor})-(\ref{Model:HKmodelbound}).
Suppose $\{\xi_i(t), i\in\mathcal{V}, t>0\}$ are i.i.d. with $\mathbb{E}\,\xi_1(t)=\textbf{0}$, $\mathbb{E}\, \xi_1^2(1)\neq\textbf{0}$ and $\delta\leq\epsilon/2$, then for all $d\geq 1$ and any given initial state $x(0)=(x_1(0),\ldots,x_n(0))\in [-1,1]^{n\times d}$, we have  $\mathbb{E}\,T<\infty$.
\end{thm}
Theorem \ref{Thm:mainbounded} establishes that in bounded spaces, the HK dynamics achieve quasi-synchronization almost surely (a.s.) within finite time for all dimensions. Moreover, the stopping time is integrable, i.e., $\mathbb{E}\,T<\infty$, highlighting the critical role of boundedness in ensuring predictable and finite-time synchronization.

To prove Theorem \ref{Thm:mainbounded}, we step a little further and consider a general system.

For a given closed convex set $K\subset \mathbb{R}^{n\times d}$, and any $x\in \mathbb{R}^{n\times d}$, denote $P_K(x)=\bar{x}\in K$ with $\| x-\bar{x}\|=\min_{y\in K}\| x-y\|$. Here, $\| \cdot\|$ is the Euclidean norm. If $x\in K$, we have $P_K(x)=x$. Since $K$ is a closed convex set, we know that for any $x\in \mathbb{R}^{n\times d}$, $P_K(x)\in K$ exists. For $z\in\mathbb{R}^{n\times d}$ and $r>0$, denote $$B_z(r)=\{x\in\mathbb{R}^{n\times d}:\| x-z\| \leq r\},$$ and let $B(r)=B_z(r)$ when $z=0$. Consider a system
\begin{equation}\label{Model:randomwalkbound}
  S(t+1)=P_{B(r)}(f(S(t))+\xi(t+1)),
\end{equation}
for $ r>0$, where $f=(f_1,\ldots,f_n)$ with $f_i:\mathbb{R}^{n\times d}\rightarrow \mathbb{R}^d$, and $\xi_i(t)\in \mathbb{R}^d, i\in\mathcal{V}=\{1,\ldots,n\}, t\geq 1$ are a sequence of zero-mean i.i.d. random variables.


We assume the function $f$ in (\ref{Model:randomwalkbound}) satisfies the following condition:
there exists $\alpha>0$ and a nonempty closed convex subset $D\subset \mathbb{R}^{n\times d}$ such that for any $x\in \mathbb{R}^{n\times d}-D$,
\begin{equation}\label{Condition:f}
\| f(x)-P_D(f(x))\|\leq \alpha \| x-P_D(x)\|.
\end{equation}
When $0<\alpha< 1$, $f$ is a transformation which shortens the distance of $x\in\mathbb{R}^{n\times d}$ to $D$.


Also, we assume the noise $\{\xi(t), t\geq 1\}$ in (\ref{Model:randomwalkbound}) has a uniform positive joint probability density, i.e., there exists a constant $\rho>0$ such that
\begin{equation}\label{prodensi}
  \mathbb{P}\{\xi(t)\in B\}\geq \rho \mu(B)
\end{equation}
for any $t>0$ and Lebesgue measurable set $B\subseteq B(r)$,  where $\mu$ is the Lebesgue measure.
The following lemma shows that, the noise-driven system in bounded space can achieve any given region of the state space within an integrable time under some interacting conditions.

\begin{lem}\label{Lem:mainbounded}
Consider the system  (\ref{Model:randomwalkbound}).
Suppose that $\{\xi_i(t), i\in\mathcal{V}, t>0\}$ are i.i.d. random variables, and the inequalities   (\ref{Condition:f}) and (\ref{prodensi}) are satisfied. For any $r_0\in(0,r)$, let $D=B(r_0)$ and
\begin{equation}\label{Equa:defofT}
  T_D=\inf\{t\geq 1:S(t)\in D\},
\end{equation}
then for any initial state $S(0)\in B(r)$, we have
$\mathbb{E}\,T_D<\infty.$
\end{lem}
\begin{proof}
Clearly, if $S(0)\in D$,
\begin{equation}\label{Equa:Tequ1}
  \mathbb{P}\{T_D=1\}=1.
\end{equation}
Now we consider the case when $S(0)\in B(r)\setminus D$. Denote $O_t=P_D({S(t)})$, $\hat{S}(t)=f(S(t))$, $t\geq 0$, then when
\begin{equation}\label{Cond:xi1}
  \xi(1)\in B_{P_D(S(0))}(\lambda \|\hat{S}(0)-O_0\|)-\hat{S}(0)
\end{equation}
with $0<\lambda< \frac{1}{\alpha}$, by the definition of $P_D(\cdot)$, (\ref{Condition:f}) and (\ref{Cond:xi1}),
\begin{equation}\label{Equa:disderea}
\begin{split}
  \| S(1)-O_1\|=&\| S(1)-P_D(S(1))\|\leq \| S(1)-P_D(S(0))\|\\
  =&\| f(S(0))+\xi(1)-P_D(S(0))\|\leq \lambda\| f(S(0))-O_0\|\\
  \leq& \lambda\alpha\| S(0)-O_0\|.
\end{split}
\end{equation}
If $\lambda\alpha\| S(0)-O_0\|\leq r_0$, we know $S(1)\in D$ under (\ref{Cond:xi1}). Otherwise, repeating the above process and taking
\begin{equation*}
  \xi(t)\in B_{P_D(S(t-1))}(\lambda \| \hat{S}(t-1)-O_{t-1}\|)-\hat{S}(t-1),
\end{equation*}
we obtain
\begin{equation*}
\begin{split}
  \| S(t)-P_D(S(t))\|\leq&\| S(t)-P_D(S(t-1))\|\\
  =&\| f(S(t-1))+\xi(t)-P_D(S(t-1))\|\\
  \leq &\lambda \| f(S(t-1))-O_{t-1}\|\\
  \leq &\lambda\alpha\| S(t-1)-O_{t-1}\|\\
  \leq& (\lambda\alpha)^t\| S(0)-O_0\|,\quad t\geq 1.
\end{split}
\end{equation*}
Take a sufficiently small $\phi\in(0, \frac{r_0}{2}]$ and denote $L=\lceil\log_{\lambda\alpha} \frac{\phi}{2\alpha r}\rceil+1$.
Since $D$ is a closed convex set, by the definition of $S_D(f(S(L-1)))$, we have
\begin{equation}\label{Equa:distdelta}
\begin{split}
  \| f(S(L-1))-P_D(f(S(L-1)))\|\leq&\alpha\| S(L-1)-O_{L-1}\|\leq \alpha(\lambda\alpha)^{L-1}\| S(0)-O_0\|\\
  \leq&\phi.
  \end{split}
\end{equation}
Consider the set
\begin{equation*}
  B_L=-B_{S_D(f(S(L-1)))}(\phi),
\end{equation*}
then when $\xi(L)\in B_L$, from (\ref{Equa:distdelta}) we have
\begin{equation}\label{Equa:SLinD}
\begin{split}
 \| S(L)\|=\|f(S(L-1))+\xi(L)\|\leq& \|f(S(L-1))-S_D(f(S(L-1)))\|+\|S_D(f(S(L-1)))+\xi(L)\|\\
  \leq&\phi+\phi=2\phi\leq r_0.
\end{split}
\end{equation}
Then $S(L)\in D$.
Denote $O_t^f=O_{t-1}-\hat{S}(t-1)$,
by the independence of $\xi(t), t\geq 1$ and (\ref{prodensi}), there exists $p>0$ such that for $t\leq L$,
\begin{equation}\label{Equa:probnoise1}
  \mathbb{P}\Big\{\xi(t)\in B_{O_t^f}(\lambda \|\hat{S}(t-1)-O_{t-1}\|)\Big\}\geq p,
\end{equation}
and also
\begin{equation}\label{Equa:probnoise2}
\mathbb{P}\{\xi(L)\in B_L\}\geq p,
\end{equation}
then by (\ref{Equa:Tequ1}), (\ref{Equa:SLinD}) and the independence of $\xi(t), t\geq 1$, we have for any $S(0)\in B_o(r)$,
\begin{equation}\label{Equa:probofreach}
\begin{split}
  \mathbb{P}\{S(L)\in D\}\geq&\prod\limits_{t<L}\mathbb{P}\Big\{\xi(t)\in B_{O_t^f}(\lambda \|\hat{S}(t-1)-O_{t-1}\|)\Big\}\mathbb{P}\{\xi(L)\in B_L\}\\
  \geq& p^{L}>0.
  \end{split}
\end{equation}
Hence by (\ref{Equa:probofreach}) and Morkov property, we have
\begin{equation}\label{Equa:probTmL}
\begin{split}
  \mathbb{P}\{T_D\geq mL\}=& \mathbb{P}\{S(t)\notin D, t< mL\}\\
\leq&\prod\limits_{k=1}^{m}\mathbb{P}\{S(kL+1)\notin D|S((k-1)L)\notin D\}\\
  \leq& (1-p^{L})^m, \quad m\geq 1.
  \end{split}
\end{equation}
Then by (\ref{Equa:probTmL}), we get
\begin{equation*}
\begin{split}
  \mathbb{E}\,T_D=&\sum\limits_{t=1}^\infty\mathbb{P}\{T_D\geq t\}\\
  =&\sum\limits_{m=0}^\infty\bigg(\mathbb{P}\{T_D\geq mL\}+\sum\limits_{t=mL+1}^{mL+L}\mathbb{P}\{T_D\geq t\}\bigg)\\
  \leq&(L+1)\sum\limits_{m=0}^\infty\mathbb{P}\{T_D\geq mL\}\\
  \leq &(L+1)\sum\limits_{m=0}^\infty(1-p^L)^m\\
  =&\frac{L+1}{p^{L}}<\infty.
  \end{split}
\end{equation*}
This completes the proof.
\end{proof}

\begin{rem}
In the proof, the assumption allows $\alpha>1$ which means $f$ can be an expansive mapping. The lemma shows that even in this case, noise satisfying some conditions (here is condition (\ref{prodensi})) can still control the systems into any smaller region in the space. Since $\lambda<\frac{1}{\alpha}$, larger $\alpha$ results in smaller $\lambda$, which means smaller region of noise by (\ref{Cond:xi1}). Furthermore, if $\alpha=1$, a simpler noise condition can be use to obtain the conclusion of Lemma \ref{Lem:mainbounded}.
\end{rem}
\begin{cor}\label{Cor:syninboundspace}
Suppose $\alpha=1$ in (\ref{Condition:f}), $\{\xi(t), t\geq 1\}$ are nondegenerate i.i.d. random variables with $\mathbb{E}\,\xi(1)=0$ and $\|\xi_i(t)\|\leq \frac{r_0}{2}$ a.s. Then for any initial state $S(0)\in B(r)$ of the system (\ref{Model:randomwalkbound}),
\begin{equation*}
 \mathbb{E}\,T_D<\infty, ~\text{for all}~ d\geq 1.
\end{equation*}
\end{cor}
\begin{proof}
When $\alpha=1$,
\begin{equation*}
  \| S(1)-O_1\|\leq \lambda\| S(0)-O_0\|
\end{equation*}
holds for any $0<\lambda<1$ (refer to (\ref{Equa:disderea})). In other words, the distance of $S(t)$ to $D$ is nonincreasing with $t$ even without noise. Since $\xi_i(t), i\in\mathcal{V}, t\geq 1$ are nondegenerate i.i.d. random variables, it is easy to show that there exists $0<p<1$ such that (\ref{Equa:probnoise1}) and (\ref{Equa:probnoise2}) hold. The rest of the proof completely copies the proof of Lemma \ref{Lem:mainbounded}.
\end{proof}

\emph{Proof of Theorem \ref{Thm:mainbounded}:}
Rewrite the bounded system (\ref{Model:HKmodelbound}) as
\begin{equation}\label{Model:HKboundrewrite}
  x(t+1)=P_{B(\epsilon/2)}(F(x(t))+\xi(t+1)).
\end{equation}
Let $D=B(\epsilon/2)$, by equation (\ref{Model:HKmodelbound}), we have
\begin{equation*}
  \|F(x(t))-P_D(F(x(t)))\|\leq \|x(t)-P_D(x(t))\|.
\end{equation*}
Then the conclusion follows from Corollary \ref{Cor:syninboundspace}.\hfill $\Box$

\subsection{Case of unbounded space}\label{Subsec:unboundedspace}
The study of unbounded space is of significant importance, both theoretically and practically. On the one hand, as a fundamental mechanism of self-organization, investigating the behavior of HK dynamics in unbounded space helps to theoretically elucidate the inherent capabilities and properties of this mechanism in the absence of additional constraints, thereby holding substantial scientific value. On the other hand, research on unbounded scenarios provides insights into the behavior of systems when the spatial scale is extremely large. Furthermore, as demonstrated in our study of the one-dimensional case, analyzing unbounded systems imposes higher mathematical demands compared to bounded systems, offering a more rigorous framework for understanding complex dynamics.

In this part, we will study the random stopping time $T$ for the system in unbounded space, where we observe that the properties of $T$ differ significantly from those in bounded space. Specifically, in unbounded space, the system exhibits unique phenomena such as non-integrable stopping times and probabilistic synchronization, which are not observed in bounded settings. These findings not only deepen our understanding of HK dynamics but also highlight the critical role of spatial boundedness in shaping the emergence of self-organizing order.

\begin{thm}\label{Thm:mainunbounded}
Suppose $\{\xi_i(t), i\in\mathcal{S}, t>0\}$ are i.i.d. with $\mathbb{E}\,\xi_1(t)=\textbf{0}$, $\mathbb{E}\, \xi_1^2(1)\neq\textbf{0}$ and $\delta\leq\epsilon/2$, then for the unbounded system (\ref{Model:neighbor}) and (\ref{Model:HKmodelunbound}), we have for $d=1,2$,
\begin{equation*}
 (a)\begin{split}
 &\mathbb{P}\{T<\infty\}=1, ~\text{for all}~x(0)\in (-\infty, \infty)^{n\times d}\\
 &\mathbb{E}\,T=\infty,~\text{for some}~x(0)\in (-\infty, \infty)^{n\times d}.
 \end{split}
\end{equation*}
Furthermore, when $\{\xi_i(t), i\in\mathcal{V}, t\geq 1\}$ are symmetric, i.e., $\xi_i(t)$ and $-\xi_i(t)$ have a same distribution, then for $d\geq 3$ and some initial states $x(0)\in (-\infty, \infty)^{n\times d}$, we have
\begin{equation*}
 (b)\begin{split}
 &\mathbb{P}\{T<\infty\}<1, \\
 &\mathbb{E}\,T=\infty.
\end{split}
\end{equation*}
\end{thm}
\begin{rem}
A stopping time $T$ is said to be integrable if $\mathbb{E}\,T<\infty$. From Proposition \ref{Prop:quasisynHK}, there exist certain initial states-for example, where all individuals are mutual neighbors-from which the systems can achieve quasi-synchronization immediately. Consequently, $\mathbb{E}\,T<\infty$ in (a) and $\mathbb{P}\{T<\infty\}=1$ in (b). However, here we consider arbitrary initial states, and our conclusions imply that there exist initial configurations for which $\mathbb{E}\,T=\infty$ in (a) and $\mathbb{P}\{T<\infty\}<1$ in (b). In other words, for one- and two-dimensional unbounded spaces, the system (\ref{Model:HKmodelunbound}) can almost surely achieve quasi-synchronization in finite time from any initial state, even though the stopping time is not integrable. In contrast, for higher-dimensional spaces, the system can no longer achieve quasi-synchronization almost surely in finite time, except for some very specific initial states.
Moreover, the non-integrability of the stopping time in (a) indicates that, compared to bounded spaces where the stopping time is integrable, the time required to achieve quasi-synchronization in low-dimensional unbounded spaces becomes essentially longer, even though it remains finite.
\end{rem}
The proof of Theorem \ref{Thm:mainunbounded} relies heavily on analyzing the properties of some random walks. First we give some lemmas about one-dimensional random walks, which will play an important role in proving the conclusion (a) of Theorem \ref{Thm:mainunbounded}.
\begin{lem}\cite{Chow1997}\label{Lem:Tboundedwalk}
Suppose $\{\xi(t), t\geq 1\}$ are one-dimensional nondegenerate i.i.d. random variables with $\mathbb{E}\,\xi(1)=0$, let
\begin{equation*}
  U(t)=\sum_{j=1}^t\xi(j),\,\text{and}\,\,\, T_U=\inf\{t\geq 1: U(t)\leq b\}
\end{equation*}
for $b\leq 0$, then $\mathbb{E}\,T_U=\infty$.
\end{lem}
Lemma \ref{Lem:Tboundedwalk} shows that for a random walk of one-dimension, its stopping time of arriving any point is non-integrable. Moreover, we can prove that when the random walk is stretched via some functions, the same conclusion holds.
\begin{lem}\label{Lem:Tgeneralrw}
Let $U(t)=\sum_{i=1}^{t}\xi(i)$, $S(1)=\xi(1), S(t+1)=g(U(t))+h(\xi(t+1)), t\geq 1$, where $g, h$ satisfy $|g(x)|\geq \beta|x|$ with some $\beta>0$, $|h(x)|\leq M<\infty$ for some $M>0$, and $g(0)=h(0)=0$, $g(x)x>0, h(x)x> 0$ for $x\neq 0$. Denote $T_1=\inf\{t\geq 1: S(t)\leq 0\}$, then $\mathbb{E}\,T_1=\infty$.
\end{lem}
\begin{proof}
Since $\{\xi_i(t), i\in\mathcal{V}, t\geq 1\}$ are nondegenerate i.i.d. variables with $\mathbb{E}\,\xi(1)=0$, it is easy to show that there exists $a>0, 0<p<1$ such that $P\{\xi_i(t)> a\}= p$. Let $L=\lceil\frac{M}{a\beta}\rceil$, $U_t^k=\sum_{i=t+1}^{t+k}\xi(i)$.

If $\xi(1)>a, \ldots, \xi(L)>a, \min_{1\leq k\leq t-L}U_{L}^k>0$, then for $L<j\leq t$, we have $U(j)>0$, hence $g(U(j))\geq \beta U(j)>0$, and
\begin{equation}\label{Equa:SjgeqL}
\begin{split}
  S(j)=&g(U(j-1))+h(\xi(j))\geq \beta U(j-1)-M\\
  =&\beta(U_L+U_L^{j-L})-M\\
  \geq& \beta aL+\beta U_L^{j-L}-M>0,\quad L< j\leq t.
\end{split}
\end{equation}
While for $1\leq j\leq L$, since $U(j)>aj>0$, by the properties of $g, h$, we know $g(U(j))>0, h(\xi(j))>0$, and thus
\begin{equation}\label{Equa:SjleqL}
\begin{split}
  S(j)=g(U(j-1))+h(\xi(j))>0,\quad 1\leq j\leq L.
\end{split}
\end{equation}
From (\ref{Equa:SjgeqL}) and (\ref{Equa:SjleqL}), we have for $t>L$,
\begin{equation}\label{Equa:UsubsetQ}
\begin{split}
  \Big\{\xi(1)>a, \ldots, \xi(L)>a, \min_{1\leq k\leq t-L}U_L^k>0\Big\}\subset\Big\{\min_{1\leq j\leq t} S(j)>0\Big\}.
  \end{split}
\end{equation}
By independence and Markov property, we have
\begin{equation}\label{Equa:probUmax}
\begin{split}
  \mathbb{P}\Big\{\xi_1>a, \ldots, \xi_L>a, \min_{1\leq k\leq t-L}U_L^k>0\Big\}=&p^L\mathbb{P}\Big\{\min_{1\leq k\leq t-L}U_{L}^k>0\Big\}\\
  =&p^L\mathbb{P}\Big\{\min_{1\leq k\leq t-L}U_k>0\Big\},
\end{split}
\end{equation}
then by (\ref{Equa:UsubsetQ}) and (\ref{Equa:probUmax}), we have for $t>L$,
\begin{equation}\label{Equa:probUlessprobQ}
  p^L\mathbb{P}\Big\{\min_{1\leq k\leq t-L}U_k>0\Big\}\leq \mathbb{P}\Big\{\min_{1\leq j\leq t} S(j)>0\Big\}.
\end{equation}
Denote $T_0=\inf\{t\geq 1: U(t)\leq 0\}$, then by Lemma \ref{Lem:Tboundedwalk},
\begin{equation}\label{Equa:a30}
  \mathbb{E}\,T_0=\infty.
\end{equation}
Hence by (\ref{Equa:probUlessprobQ}) and (\ref{Equa:a30}),
\begin{equation*}
  \begin{split}
  \mathbb{E}\,T_1=&\sum\limits_{t=1}^\infty\mathbb{P}\{T_1\geq t\}=1+\sum\limits_{t=1}^\infty\mathbb{P}\Big\{\min_{1\leq j\leq t} S(j)>0\Big\}\\
  \geq&\sum\limits_{t=L+1}^\infty\mathbb{P}\Big\{\min_{1\leq j\leq t} S(j)>0\Big\}\\
  \geq& p^L\sum\limits_{t=L+1}^\infty\mathbb{P}\Big\{\min_{1\leq k\leq t-L}U(k)>0\Big\}\\
  =&p^L\sum\limits_{t=1}^\infty\mathbb{P}\Big\{\min_{1\leq k\leq t}U(k)>0\Big\}\\
  =&p^L\sum\limits_{t=1}^\infty\mathbb{P}\{T_0\geq t+1\}\\
  =&p^L(\mathbb{E}\,T_0-1)=\infty.
  \end{split}
\end{equation*}
This completes the proof.
\end{proof}

Let $\{\xi(t), t\geq 1\}$ are nondegenerate i.i.d. variables in $\mathbb{R}^d$ with $\mathbb{E}\,\xi(1)=\textbf{0}$ and $U(t)=\sum_{i=1}^{t}\xi(i)$ is a random walk in $\mathbb{R}^d$, then $U(t)$ is said to be recurrent if $\lim\inf_{t\rightarrow\infty}\|U(t)\|=0$ a.s. and transient if $\|U(t)\|\rightarrow\infty$ a.s.
\begin{lem}\cite{Kallenberg2001}\label{Lem:recurrencerw}
$U(t)$ is recurrent for $d=1,2$ and transient for $d\geq 3$.
\end{lem}

\emph{Proof of Theorem \ref{Thm:mainunbounded}:}
For $d\geq 1$, we consider the following initial opinion configuration: there are two nonempty subgroups $\mathcal{V}_1, \mathcal{V}_2$, such that $\mathcal{V}_1\bigcup \mathcal{V}_2=\mathcal{V}$ and $\mathcal{V}_1\bigcap \mathcal{V}_2=\varnothing$. Additionally, $d_{\mathcal{V}_1}(0)\leq \epsilon,d_{\mathcal{V}_2}(0)\leq \epsilon$, and for any $i\in\mathcal{V}_1, j\in\mathcal{V}_2$, it satisfies $\| x_i(0)-x_j(0)\| >\sqrt{2}\epsilon+2\delta$. In other words, at initial time, the system consists of two separated subgroups. Under this initial opinion configuration,
Denote
\begin{equation*}
  T_0=\inf\limits_{t\geq 1}\Big\{t: \min\limits_{i\in\mathcal{V}_1, j\in\mathcal{V}_2}\| x_i(t)-x_j(t)\|\leq \epsilon\Big\},
\end{equation*}
then
\begin{equation}\label{Equa:T0leqT}
  T_0\leq T,\qquad a.s.
\end{equation}
and by (\ref{Model:HKmodelunbound}), we have for $t<T$, $i\in \mathcal{V}_k,\, k=1,2$,
\begin{equation}\label{Equa:xitvk}
\begin{split}
  x_i(t+1)=&\frac{1}{|\mathcal{V}_k|}\sum\limits_{j\in\mathcal{V}_k}x_j(t)+\xi_i(t+1)\\
  =&\frac{1}{|\mathcal{V}_k|}\sum\limits_{j\in\mathcal{V}_k}x_j(0)+\sum\limits_{l=1}^{t}\frac{\sum_{j\in\mathcal{V}_k}\xi_j(l)}{|\mathcal{V}_k|}+\xi_i(t+1).
  \end{split}
\end{equation}
Now for given $i\in\mathcal{V}_1, j\in\mathcal{V}_2$ and $t\geq 1$, we denote
\begin{equation}\label{Equa:dij1}
\begin{split}
  d_{ij}(t)=&\bigg(\frac{\sum_{l\in\mathcal{V}_1}x_l(0)}{|\mathcal{V}_1|}-\frac{\sum_{l\in\mathcal{V}_2}x_l(0)}{|\mathcal{V}_2|}\bigg)+
  \bigg(\sum\limits_{k=1}^{t-1}\frac{\sum_{l\in\mathcal{V}_1}\xi_l(k)}{|\mathcal{V}_1|}-\sum\limits_{k=1}^{t-1}\frac{\sum_{l\in\mathcal{V}_2}\xi_l(k)}{|\mathcal{V}_2|}\bigg)\\
  &+(\xi_i(t)-\xi_j(t))
\end{split}
\end{equation}
and
\begin{equation}\label{Equa:defiT1}
  T_1=\inf\limits_{t\geq 1}\{t: \min\limits_{i\in\mathcal{V}_1, j\in\mathcal{V}_2} d_{ij}(t)\leq \epsilon\},
\end{equation}
then by (\ref{Equa:xitvk})
\begin{equation}\label{Equa:T1leqT0}
  T_1= T_0,\qquad a.s.
\end{equation}
Let
\begin{equation*}
  y(t)=\frac{\sum_{l\in\mathcal{V}_1}\xi_l(t)}{|\mathcal{V}_1|}-\frac{\sum_{l\in\mathcal{V}_2}\xi_l(t)}{|\mathcal{V}_2|},\quad t\geq 1
\end{equation*}
and
\begin{equation*}
\begin{split}
  &Q_{ij}(1)=\xi_i(1)-\xi_j(1),\\
 &Q_{ij}(t)=\sum_{k=1}^{t-1}y(k)+(\xi_i(t)-\xi_j(t)), \quad t\geq 2,\\
  &Z(t)=\sum_{k=1}^ty(k), \,t\geq 1,
  \end{split}
\end{equation*}
then
\begin{equation}\label{Equa:QijZt}
  Q_{ij}(t+1)=Z(t)+(\xi_i(t+1)-\xi_j(t+1)),\qquad t\geq 1,
\end{equation}
and by (\ref{Equa:dij1})
\begin{equation}\label{Equa:dij2}
  d_{ij}(t)=\bigg(\frac{\sum_{l\in\mathcal{V}_1}x_l(0)}{|\mathcal{V}_1|}-\frac{\sum_{l\in\mathcal{V}_2}x_l(0)}{|\mathcal{V}_2|}\bigg)+Q_{ij}(t).
\end{equation}
Since $\mathbb{E}\,\xi_1(1)=\textbf{0}, \mathbb{E}\,\xi_1^2(1)\neq\textbf{0}$ and $\|\xi_i(1)\|\leq \delta$ a.s., we know $\{y(t), t\geq 1\}$ are i.i.d. with $\mathbb{E}\,y(t)=\textbf{0}, \mathbb{E}\,y^2(t)\neq\textbf{0}, \|y(t)\|\leq 2\delta$ a.s., $t\geq 1$.

(i) For $d=1$, by Theorem 3.1 of \cite{Su2019tac},
\begin{equation*}
  \mathbb{P}\{T<\infty\}=1.
\end{equation*}
To prove the integrability of $T$, we consider the former initial opinion configuration and suppose $x_i(0)>x_j(0)$ for any $i\in\mathcal{V}_1, j\in\mathcal{V}_2$ without loss of generality.
Denote
\begin{equation}\label{Equa:defTQ}
T_Q=\inf\limits_{t\geq 1}\Big\{t:\min\limits_{i\in\mathcal{V}_1, j\in\mathcal{V}_2}Q_{ij}(t)\leq 0\Big\},
\end{equation}
since $\frac{\sum_{l\in\mathcal{V}_1}x_l(0)}{|\mathcal{V}_1|}-\frac{\sum_{l\in\mathcal{V}_2}x_l(0)}{|\mathcal{V}_2|}>\epsilon$ by assumption, we know from (\ref{Equa:defiT1}), (\ref{Equa:dij2}) and (\ref{Equa:defTQ}) that
\begin{equation}\label{Equa:T1geqTQ}
  T_1\geq T_Q, \qquad a.s.
\end{equation}
Take $g(x)=x$ in Lemma \ref{Lem:Tgeneralrw} and consider of the boundedness of $\xi_i(t)$ (hence $h(x)$ in Lemma \ref{Lem:Tgeneralrw} is bounded),
\begin{equation}\label{Equa:TQinfty}
  \mathbb{E}\,T_Q=\infty,
\end{equation}
then by (\ref{Equa:T0leqT}), (\ref{Equa:T1leqT0}), (\ref{Equa:T1geqTQ}) and (\ref{Equa:TQinfty}), we obtain
\begin{equation*}
  \mathbb{E}\,T\geq \mathbb{E}\,T_1\geq \mathbb{E}\,T_Q=\infty.
\end{equation*}

(ii) For $d=2$, we first prove $\mathbb{P}\{T<\infty\}=1$. Considering (\ref{Equa:dij2}), and using the recurrence of random walk in $\mathbb{R}^2$ (Lemma \ref{Lem:recurrencerw}) and the homogenous $\epsilon$ of HK model (\ref{Model:HKmodelunbound}), then following a similar argument of Proposition 3.1 of \cite{Su2019tac}, we can obtain the conclusion.

Now we prove the non-integrability of $T$. Denote
\begin{equation*}
  x_i(t)=(x_i^1(t),x_i^2(t)),\, \xi_i(t)=(\xi_i^1(t),\xi_i^2(t)),\,i\in\mathcal{V}, t\geq 1.
\end{equation*}
Since
\begin{equation*}
  \min\limits_{i\in\mathcal{V}_1, j\in\mathcal{V}_2}\| x_i(0)-x_j(0)\| >\sqrt{2}\epsilon,
\end{equation*}
necessarily $\min\limits_{i\in\mathcal{V}_1, j\in\mathcal{V}_2}|x_i^k(0)-x_j^k(0)|>\epsilon$ holds for at least $k=1$ or $k=2$. Without loss of generality, we suppose $\min\limits_{i\in\mathcal{V}_1, j\in\mathcal{V}_2}|x_i^1(0)-x_j^1(0)|>\epsilon$.

Let
\begin{equation}\label{Equa:defiT1dim2}
  T_1=\inf\limits_{t\geq 1}\{t: \min\limits_{i\in\mathcal{V}_1, j\in\mathcal{V}_2} |x_i^1(t)-x_j^1(t)|\leq\epsilon\},
\end{equation}
then by the conclusion of $d=1$, we have
\begin{equation}\label{Equa:T1d2infty}
  \mathbb{E}\,T_1=\infty.
\end{equation}
Moreover, $\{d_\mathcal{V}(t)\leq\epsilon\}\subset\{\min\limits_{i\in\mathcal{V}_1, j\in\mathcal{V}_2} |x_i^1(t)-x_j^1(t)|\leq\epsilon\}$ for $t\geq 1$, hence
$T_1\leq T$ a.s., and it follows from (\ref{Equa:T1d2infty}) that
\begin{equation*}
  \mathbb{E}\,T\geq\mathbb{E}\,T_1=\infty.
\end{equation*}

(iii) For $d\geq 3$, by Lemma \ref{Lem:recurrencerw} we know $Z(t)$ in (\ref{Equa:QijZt}) is transient, and hence for any $M>0$,
\begin{equation*}
  \mathbb{P}\{\| Z(t)\|\leq M, i.o.\}<1.
\end{equation*}
By Hewitt-Savage 0-1 law, we have
\begin{equation}\label{Equa:Ztboundprob0}
  \mathbb{P}\{\| Z(t)\|\leq M, i.o.\}=0.
\end{equation}
Give any $O_1\in \mathbb{R}^{n\times d}$ with $\| O_1\|>0$, and for any $0<r<\| O_1\|$, denote
\begin{equation}\label{Equa:defT1unbound}
  T_1=\inf\{t\geq 1: Z(t)\in B_{O_1}(r)\},
\end{equation}
then we can prove
\begin{equation*}
 \mathbb{P}\{T_1<\infty\}<1.
\end{equation*}
In fact, let $O_2=O_1+Z(T_1)$, and for a random time $Y\geq 0$ a.s. denote
\begin{equation*}
  Z_Y(t)=\sum_{i=1}^ty(Y+i),
\end{equation*}
\begin{equation}
  T_2=\inf\{t\geq 1: Z(T_1)+Z_{T_1}(t)\in B_{O_2}(r)\},
\end{equation}
then by i.i.d. property of $\{y(t), t\geq 1\}$ and strong Morkov property
\begin{equation}
\begin{split}
   \mathbb{P}\{Z(T_1)+Z_{T_1}(t)\in B_{O_2}(r)\}\geq& \mathbb{P}\{T_1<\infty,Z_{T_1}(t)\in B_{O_2}(r)-Z(T_1)\}\\
   =&\mathbb{P}\{T_1<\infty\}\mathbb{P}\{Z(t)\in B_{O_1}(r)\}
\end{split}
\end{equation}
and hence
\begin{equation}\label{Equa:T2iterT1}
 \mathbb{P}\{T_2<\infty\}\geq\mathbb{P}^2\{T_1<\infty\}.
\end{equation}
Further, let $O_3=Z(T_1+T_2)-O_1$,
\begin{equation}
  T_3=\inf\{t\geq 1: Z(T_1 +T_2)+Z_{T_1+T_2}(t)\in B_{O_3}(r)\},
\end{equation}
then by the symmetry of $Z(t)$ and the strong Morkov property
\begin{equation}
\begin{split}
   \mathbb{P}\{Z(T_1+ T_2)+Z_{T_1+T_2}(t)\in B_{O_3}(r)\}\geq& \mathbb{P}\{T_1,T_2<\infty,Z_{T_1+T_2}(t)\in B_{O_3}(r)-Z(T_1+ T_2)\}\\
   =&\mathbb{P}\{T_1<\infty\}\mathbb{P}\{T_2<\infty\}\mathbb{P}\{Z(t)\in B_{-O_1}(r)\}\\
   =&\mathbb{P}\{T_1<\infty\}\mathbb{P}\{T_2<\infty\}\mathbb{P}\{-Z(t)\in B_{O_1}(r)\}\\
   =&\mathbb{P}\{T_1<\infty\}\mathbb{P}\{T_2<\infty\}\mathbb{P}\{Z(t)\in B_{O_1}(r)\}
\end{split}
\end{equation}
and hence
\begin{equation}\label{Equa:T3iterT1}
\begin{split}
 \mathbb{P}\{T_3<\infty\}\geq&\mathbb{P}\{T_1<\infty\}\mathbb{P}\{T_2<\infty\}\mathbb{P}\{T_1<\infty\}\\
 \geq&\mathbb{P}^4\{T_1<\infty\}.
 \end{split}
\end{equation}
Now, for $m\geq 2$, denote $T(m)=\sum_{i=1}^{m}T_i$, and
\begin{equation*}
  \begin{split}
  O_{2m}=O_1+Z(T(2m-1)), \,O_{2m+1}=Z(T(2m))-O_1,
  \end{split}
\end{equation*}
and
\begin{equation}
\begin{split}
  T_{2m}=&\inf\bigg\{t\geq 1: Z(T(2m-1))+Z_{T(2m-1)}(t)\in B_{O_{2m}}(r)\bigg\},\\
  T_{2m+1}=&\inf\bigg\{t\geq 1: Z(T(2m))+Z_{T(2m)}(t)\in B_{O_{2m+1}}(r)\bigg\},
  \end{split}
\end{equation}
then considering the symmetry of $Z(t)$ and the strong Morkov property, and following the argument of (\ref{Equa:T2iterT1}) and (\ref{Equa:T3iterT1}), we have
\begin{equation}\label{Equa:TniterT1}
\begin{split}
  \mathbb{P}\{T_{2m}<\infty\}\geq&\prod_{k=1}^{2m-1}\mathbb{P}\{T_k<\infty\}\mathbb{P}\{T_1<\infty\}\\
  \geq&(\mathbb{P}\{T_1<\infty\})^{L_1},\\
  \mathbb{P}\{T_{2m+1}<\infty\}\geq&\prod_{k=1}^{2m}\mathbb{P}\{T_k<\infty\}\mathbb{P}\{T_1<\infty\}\\
  \geq&(\mathbb{P}\{T_1<\infty\})^{L_2},
  \end{split}
\end{equation}
for some $L_1, L_2>0$.
If $\mathbb{P}\{T_1<\infty\}=1$, we know from (\ref{Equa:T2iterT1}), (\ref{Equa:T3iterT1}) and (\ref{Equa:TniterT1}) that
\begin{equation}\label{Equa:Tnleqinfty}
  \mathbb{P}\{T_m<\infty\}=1,\quad m\geq 1.
\end{equation}
Moreover, from the definition of $T_i$ we can obtain that
\begin{equation*}
  \mathbb{P}\bigg\{\| Z(T(m))\|\leq 2\| O_1\|\bigg\}=1,\,\,m\geq 1.
\end{equation*}
However, this contradict with (\ref{Equa:Ztboundprob0}). Hence,
\begin{equation}\label{Equa:T1lessinfty}
 \mathbb{P}\{T_1<\infty\}<1.
\end{equation}
Denote
\begin{equation*}
  O_1=\frac{\sum_{l\in\mathcal{V}_2}x_l(0)}{|\mathcal{V}_2|}-\frac{\sum_{l\in\mathcal{V}_1}x_l(0)}{|\mathcal{V}_1|},
\end{equation*}
then $\| O_1\|>\sqrt{2}\epsilon+2\delta$ by assumption, and this together with (\ref{Equa:dij2}) imply
\begin{equation*}
  \{d_\mathcal{V}(t)\leq\epsilon\}\subset\{Z(t)\in B_{O_1}(\epsilon)\},
\end{equation*}
hence by (\ref{Equa:defstopT}) and (\ref{Equa:defT1unbound}), we have
\begin{equation}\label{Equa:TgeqT1}
  T\geq T_1, \qquad a.s.
\end{equation}
and then by (\ref{Equa:T1lessinfty}),
\begin{equation}\label{Equa:Tlessinfty}
 \mathbb{P}\{T<\infty\}\leq \mathbb{P}\{T_1<\infty\}<1.
\end{equation}
Necessarily,
\begin{equation*}
  \mathbb{E}\,T=\infty.
\end{equation*}
This completes the proof.
\hfill $\Box$

\section{Discussion  and Conclusions}\label{Section:Conclusions}

In this paper, we study the synchronization of multi-dimensional HK dynamics in bounded and unbounded space of noisy environment. Via investigating the random time when the system attains quasi-synchronization, it is revealed that both the boundedness and dimension of space determine different properties of synchronization. The findings can also reveal some interesting implications for applying high-dimensional HK models in real-world social systems. For example, in social networks where opinion formation occurs in high-dimensional spaces (e.g., multiple topics or attributes), imposing bounds on range of opinion values can potentially facilitate achieving synchronization or consensus. On the other hand, when agents' opinions span an unbounded space, the role of dimension of the space becomes critical, with higher dimensions possibly leading to the prevention of consensus or synchronization.

Our results highlight the critical role of boundedness in the emergence of self-organizing order under the influence of noise. Specifically, in bounded spaces, the system can achieve quasi-synchronization in finite time for all dimensions, demonstrating how spatial constraints facilitate the transition from disorder to order in the presence of stochastic disturbances. In contrast, in unbounded spaces, the lack of spatial constraints makes synchronization more challenging, particularly in higher dimensions, where noise can only induce synchronization with a positive probability. These findings underscore the importance of boundedness as a fundamental mechanism for the emergence of order in self-organizing systems driven by random factors.

Our results have important implications for fields where the dynamics of networks of interacting systems are ubiquitous, such as social, economic, and natural sciences. Furthermore, our findings may pave the way for future studies focused on the noise-based control of complex networked systems in high-dimensional spaces.

\end{document}